\documentclass[12pt]{article}

\usepackage{amssymb,amsmath,amsfonts,amssymb}
\usepackage{graphics,graphicx,color}

\def\@abssec#1{\vspace{.05in}\footnotesize \parindent .2in
{\bf #1. }\ignorespaces}

\graphicspath{{/EPSF/}{Figures/}}
\DeclareGraphicsExtensions{.eps}

\newtheorem{theorem}{Theorem}[section]
\newtheorem{lemma}[theorem]{Lemma}
\newtheorem{proposition}[theorem]{Proposition}

\newtheorem{remark}[theorem]{Remark}

\def \Rm {\mathbb R}

\newcommand{\eps}{\varepsilon}

\newcommand{\dps}{\displaystyle}

\newcommand{\dint}{\displaystyle\int}

\newcommand{\cout}[1]{}

\newcommand{\fea}{{f^\eps_\alpha}}

\newcommand{\vny}{{v\cdot\nabla_y}}
\newcommand{\vnx}{{v\cdot\nabla_x}}
\newcommand{\pe}{{\psi^\eta}}
\newcommand{\pes}{{\psi^{\eta*}}}
\newcommand{\pea}{{\psi^\eta_\alpha}}
\newcommand{\peas}{{\psi^{\eta*}_\alpha}}
\newcommand{\Qea}{{Q^\eta_\alpha}}
\newcommand{\bQea}{{{\cal Q}^\eta_\alpha}}
\newcommand{\Tea}{{T^\eta_\alpha}}

 \renewcommand{\arraystretch}{1.5}

\title{A corrector theory for diffusion-homogenization limits of linear transport equations\thanks{ NBA and MP were partially supported by  ANR project BLAN07-2 212988. GB was partially supported by NSF Grant DMS-0804696. GB would like to thank the Universit\'e Paul Sabatier for their hospitality in the Spring of 2009, when part of this research effort was conducted.}}

\author{Guillaume Bal \thanks{Department of Applied Physics and 
        Applied Mathematics, Columbia University, 
        New York NY, 10027; gb2030@columbia.edu} \and Naoufel Ben Abdallah \and Marjolaine Puel \thanks{Institut de Math\'ematiques, Universit\'e de Toulouse and CNRS, Universit\'e Paul Sabatier, 31062 Toulouse Cedex 9, France; {\tt puel@math.univ-toulouse.fr}.}}

\begin{document}
 
\maketitle


\begin{abstract}
   This paper concerns the diffusion-homogenization of transport equations when both the adimensionalized scale of the heterogeneities $\alpha$ and the adimensionalized mean-free path $\eps$ converge to $0$. When $\alpha=\eps$, it is well known that the heterogeneous transport solution converges to a homogenized diffusion solution. We are interested here in the situation where $0<\eps\ll\alpha\ll1$ and in the respective rates of convergences to the homogenized limit and to the diffusive limit. Our main result is an approximation to the transport solution with an error term that is negligible compared to the maximum of $\alpha$ and $\frac\eps\alpha$. After establishing the diffusion-homogenization limit to the transport solution, we show that the corrector is dominated by an error to homogenization when $\alpha^2\ll\eps$ and by an an error to diffusion when $\eps\ll\alpha^2$.
   
   Our regime of interest involves singular perturbations in the small parameter $\eta=\frac\eps\alpha$. Disconnected local equilibria at $\eta=0$ need to be reconnected to provide a global equilibrium on the cell of periodicity when $\eta>0$. This reconnection between local and global equilibria is shown to hold when sufficient {\em no-drift} conditions are satisfied. The Hilbert expansion methodology followed in this paper builds on corrector theories for the result developed in \cite{NBAPuVo}.
\end{abstract}
 

\renewcommand{\thefootnote}{\fnsymbol{footnote}}
\renewcommand{\thefootnote}{\arabic{footnote}}

\renewcommand{\arraystretch}{1.1}





\section{Setting of the problem}
 
This paper studies the interaction between the convergence to a homogenized limit and the convergence to a diffusive limit in the context of linear transport equations, or linear Boltzmann equations \cite{Ce,BaGoPeSe}, that model the propagation of particles in scattering environments. These two phenomena have been widely studied in the past \cite{BeLiPa,Albook,Ev, LaKe,La1,La2,Po,DeGoPo}. The homogenization limit typically arises when the underlying coefficients oscillate at a scale $\alpha\ll1$ much smaller than the macroscopic scale at which phenomena are observed. The diffusion limit arises in highly scattering environments in which an equilibrium emerges in the velocity variable and a diffusion equation models the spatial behavior of the probability density of particles. High scattering is modeled by a mean free path $\eps\ll1$, where $\eps$ is the (adimensionalized) main distance between successive collisions with the underlying medium.

We assume here that the oscillations in the medium are periodic with period $\alpha\ll1$. When both $\alpha$ and $\eps$ converge to $0$, we expect the transport solution to converge to the solution of a homogenized diffusion equation. How fast convergence may occur and what are the main contributions to the error between the heterogeneous transport solution and its homogenized limit are the main problems of interest in this paper. We restrict ourselves to the case where the mean free path is (much) smaller than the correlation length $\alpha$. 
We thus set $\eta=\frac \eps \alpha$ and we assume that $\eta \ll1$. The equation for the particle density takes the form:
\begin{equation}
  \label{eq:trepsalpha}
  v\cdot\nabla \fea + \eps \fea + \dfrac 1\eps Q_\alpha \fea = \eps S_\alpha(x,v) := \eps S(x,\frac x\alpha,v),
\end{equation}
in an infinite domain $\Rm^d\times V$ where $V$ is a smooth, symetric, compact domain in the velocity space such that $0\not\in V$. This equation could also be seen as an evolution equation by a change of variables $u^{\eps}_\alpha=e^{-t}\fea$ in which the source $S$ would play the role of an initial condition. We restrict ourselves to the time independent setting and to the case of constant absorption to simplify. 

The source term $S$ is $1-$periodic in $y=\frac x\alpha$. We define the collision operator as
\begin{equation}
  \label{eq:Qalpha}
  Q_\alpha f = \Sigma(\dfrac x\alpha,v) f(x,v) - \dint_V \sigma(\dfrac x\alpha,v\rq{},v)f(x,v')d\nu(v').
\end{equation}
The choice of $\Sigma$ is such that the above operator is conservative in the following sense. Define
\begin{equation}
  \label{eq:Q}
  Qf(y,v) = \Sigma(y,v) f(y,v) - \dint_V \sigma(y,v',v) f(y,v') d\nu(v'),\quad y\in Y,
\end{equation}
where $Y=[0,1]^d$ is the unit cell. Then we assume the existence of $C^{-1}\geq\psi^\eta(y,v)\geq C>0$ such that 
\begin{equation}
  \label{eq:psieta}
  (\eta \vny + Q) \pe =0 ,\qquad Y\times V.
\end{equation}
We also define
\begin{equation}
  \label{eq:psietas}
  (-\eta\vny+Q^*)\pes =0,\qquad Y\times V,
\end{equation}
where $Q^*$ is the adjoint operator to $Q$ defined for a.e. $y\in Y$ and $g\in L^2(V)$ as 
\begin{equation}
  \label{eq:Qstar}
  Q^* g (v) = \Sigma(y,v) g(v) - \dint_V \sigma(y,v,v') g(v') d\nu(v').
\end{equation}

The simultaneous limit when the mean free path and the correlation length go to zero together with $\alpha=\eps$ has been considered in several recent papers  \cite{AlBa,Ba1,Ba2,Ba3,BATa,GoMe1,GoMe2,GoPo1,GoPo2,Se}.
In all those paper, only the case $\eta=1$ is considered except in \cite{Se}, where both $\psi^\eta$ and $\psi^{\eta*}$ are space independent and hence also independent of $\eta$. A first result in the case $\eta\ll1$ was obtained in \cite{NBAPuVo} for spatially dependent $\psi^\eta$ in the setting where $\psi^{*,\eta}\equiv1$. In that paper,  a two scale convergence result is established in the setting where the limiting behavior involves the homogenization of a heterogeneous diffusion. 
The proof of that result is based on the weak formulation of the transport equation and on the method of moments. 

The aim of the present work is to implement the Hilbert expansion method, which leads to a strong convergence result under appropriate smoothness conditions in the (restricted) setting where $\psi^{*,\eta}=\psi^*(v)$. This strong convergence has two  advantages. First, we are able to compute the correction terms in the expansion of the transport solution up to a term that is negligible compared to the maximum of $\alpha$ and $\eta=\frac\eps\alpha$. Second, we expect this result also to be the first step in the study of the nonlinear Boltzmann equation given by the Fermi-Dirac equation in the spirit of \cite{NBACh1,NBAChSc}. Note that in the case where $\eta$ goes to infinity, we expect the limit to be the diffusion approximation to the homogenized transport equation obtained in \cite{DuGo}.

The main mathematical difficulty of the present work arises because the limit $\eta\to0$ in \eqref{eq:psieta} is singular. In the limit $\eta=0$, local equilibria are obtained as a function of $v\in V$ for {\em each} position $y\in Y$. For $\eta>0$, a global equilibrium on $Y\times V$ emerges, as in the standard procedure obtained when $\eta=1$ \cite{AlBa,GoPo1,GoMe1,La1,La2,Se}. The passage from the local equilibria to the global equilibrium may in fact be arbitrarily complicated. Several conditions need to be imposed in order for a well-defined macroscopic equilibrium to arise. Even in the case $\eta=1$ do we need to impose a {\em no-drift} condition. In the presence of drift, advection dominates scattering and entirely different phenomena arise (see, e.g., \cite{GoMe1}, \cite{Ba3} and the derivation of Euler equations when advection is dominant). Under appropriate sufficient symmetry assumptions similar to (though more constraining than) those in \cite{AlBa}, we are able to verify the necessary no-drift conditions used in our derivation.

The rest of the paper is structured as follows. The main hypotheses of regularity and no-drift as well as the main results of this paper are described in section \ref{sec:main}. The main result is Theorem \ref{th:main} below. The rest of the paper is devoted to its proof. Global a priori estimates are formulated and proved in section \ref{sec:apriori}. The expansion of the transport solution in powers of $\eps$ is treated in section \ref{sec:eps}. The expansion in $\eta$ of several cell solutions is given in section \ref{sec:aux}. The definition and  $\eta$-dependence of spatial density terms are given in section \ref{sec:density}. These results are combined to finish the proof of the theorem in section \ref{sec:proof}. Several technical results obtained in \cite{NBAPuVo} and the generalization of their proofs if necessary are collected in the Appendix.


\section{Main result}
\label{sec:main}

Under regularity assumptions recalled below, a standard application of a Banach fixed point theorem  \cite{DaLi} ensures that (\ref{eq:trepsalpha}) admits a unique solution  in $H^k(\Rm^d,L^2(V))$. The limit of $f^{\eps,\eta}(x,v):=f^\eps_\alpha(x,v)$ as $\eta\to0$, however, involves singular perturbations. The reason is that the local equilibria in \eqref{eq:psieta}  and \eqref{eq:psietas} become degenerate in the limit $\eta\to0$. In this limit, equilibria at different points $y\in Y$ become disconnected and this can result in a very large effect at the macroscopic scale $x$. We consider here situations where the equilibria remain smooth in the $y$ variable and generate no {\em drift}. Drift effects are ubiquitous in the homogenization of transport equations, with drastic effects as may be seen in, e.g., \cite{Ba3}. Diffusion limits arise under sufficient {\em no-drift} conditions as in, e.g., \cite{AlBa}, which ensure that transport is not in an advection-dominated regime. Our analysis in this paper shows that diffusion-like regimes are still valid in the singular limit $\eta\to0$ under appropriate no-drift assumptions. We show that these assumptions are consequences of symmetries of the transport coefficients, which we now define.

Our first main assumptions is that $\psi^{*,\eta}(y,v)$ is independent of $y$ and hence of $\eta$. Note that $Q$ in \eqref{eq:psieta} is independent of $\eta$ and it is therefore not clear why non-trivial solutions would exist for all values of $\eta$. Here, we assume that a.e. $y\in Y$, we have an equilibrium $\psi^*(v)$ solution of 
\begin{equation}
  \label{eq:psis} Q^*(\psi^*)=0 \qquad \mbox{ in } V.
\end{equation}
The equilibrium solution is allowed to depend on $v$ but is independent of $y\in Y$.
When $\sigma(y,v',v)$ and $\Sigma(y,v)$ are continuous functions bounded above and below by positive constants, then $\psi^*$ can be chosen positive and normalized so that $\int_V (\psi^*)^2(y,v) d\nu(v)=1$ a.e. $y\in Y$. Moreover, up to normalization, $\psi^*(v)$ is the unique solution to \eqref{eq:psis} as an application of the Krein Rutman theorem \cite{AlBa,DaLi} for the compact operators defined for a.e. $y\in Y$:
\begin{equation}\label{eq:compact} 
   f \mapsto \dfrac1{\Sigma(y,v)} \dint_V \sigma(y,v,v') f(v') d\nu(v'),
\end{equation}
which preserve the (solid) cone of positive continuous functions. 

The no-drift conditions mentioned above will be verified under sufficient symmetry assumptions. We first assume that:
\begin{equation}
  \label{eq:symsigma}
  \sigma(y,v',v)=\sigma(y,-v',-v)=\sigma(-y,v',v),\qquad \Sigma(y,v)=\Sigma(y,-v)=\Sigma(-y,v).
\end{equation}
We deduce that  $\psi^*(-v)=\psi^*(v)$. Note that a method to construct local equilibria consists of selecting $\sigma$ satisfying the above symmetries, $\psi^*$ arbitrary (uniformly positive) such that $\psi^*(-v)=\psi^*(v)$, and finally define $\Sigma(y,v)$ by \eqref{eq:psis}, which also satisfies \eqref{eq:symsigma}.

We also obtain the existence of a unique, bounded, positive, solution $\psi(y,v)$ of the adjoint equation
\begin{equation}
  \label{eq:psi} Q(\psi(y,\cdot))=0, \qquad \mbox{ in } V,
\end{equation}
normalized so that $\int_V \psi(y,v)\psi^*(y,v)d\nu(v) =1$ a.e. $y\in Y$. It is not difficult to observe that $\psi(y,-v)=\psi(y,v)=\psi(-y,v)$ when \eqref{eq:symsigma} holds. 

\cout{For the operators $Q$ and $Q^*$, we have the following {\em dissipation} properties
\begin{equation}
\label{eq:dissipQ}
||f-\psi\int_V f \psi^*d\nu(v) ||^2_{L^2(V)}\leq \int_V Q(f)\frac{f\psi^*}{\psi}d\nu(v),
\end{equation}
\begin{equation}
\label{eq:dissipQ*}
||f-\psi^*\int_V f \psi d\nu(v) ||^2_{L^2(V)}\leq \int_V Q^*(f)\frac{f\psi}{\psi^*}d\nu(v),
\end{equation}
which hold for each $f\in L^2(V)$ and a.e. $y\in Y$ and come from the following standard point-wise equality
$$
Q(f)\frac{f\psi^*}{\psi} (v)=-\frac{1}{2}\int_V \sigma(y,v',v)\psi(y,v)\psi^*(v')\Big|\frac{f}{\psi}(v)-\frac{f}{\psi}(v')\Big|^2 d\nu(v')+Q(\frac{f^2}{\psi})\frac{\psi^*}{2}(v).
$$}
The Krein Rutman theorem for \eqref{eq:compact} (all eigenvalues not equal to $1$ have modulus strictly smaller than $1$) shows that $0$ is a simple eigenvalue associated to $Q$ and $Q^*$ and that all other eigenvalues of $Q$ and $Q^*$ have strictly positive real part \cite{AlBa,DaLi}. On the vector space of functions $f\in L^2(V)$ such that $\int_V f(v)\psi^*(v)d\nu(v)=0$, we define
\begin{equation}
  \label{eq:Qinv}
  Q^{-1}f  := -\dint_0^\infty e^{-rQ} f dr,
\end{equation}
which converges a.e. $y\in Y$ thanks to the spectral gap we just mentioned. Note that $(Q^{-1}f,\psi^*)_{L^2(V)}=0$ by construction. The inverse operator $Q^{-*}:=(Q^*)^{-1}$ is defined similarly. We verify that under \eqref{eq:symsigma}, both $Q^{-1}$ and $Q^{-*}$ preserve the subspaces of even and odd functions in the variable $v$.

With $\psi^*$ seen as a normalized solution of $(-\eta \vny + Q^*)\psi^*=0$ on $Y\times V$, we also obtain the existence of unique, bounded, positive, solutions $\psi^{\eta}(y,v)$ of \eqref{eq:psieta} normalized such that $\int_{Y\times V} \psi^{\eta}(y,v)\psi^*(v) d\nu(v) =1$.
Upon introducing $T^\eta=\eta \vny+Q$, we also define the inverse operator
\begin{equation}
  \label{eq:Tinv}
  T^{\eta-1} f= -\dint_0^\infty e^{-rT^\eta} f dr,\qquad f\in L^2(Y\times V) \,\,\mbox{ s.t. } \dint_{Y\times V}\hspace{-.3cm} f(y,v) \psi^*(v) dy d\nu(v)=0.
\end{equation}
\cout{OK a enlever. Je suis d'accord pour Krein Rutmann. Pour le papier avec Michael, je lui avais proose les deux solutions, Krein Rutmann ou la construction a la main, il avait choisi la construction a la main. }
\cout{ As we mention above, no-drift conditions are necessary to ensure that a diffusive macroscopic equilibrium emerges in the limit $\eta\to0$. We assume the  symmetry condition  $\sigma(-y ,v ,v')=\sigma(y  ,v ,v' )$ and the reciprocity condition $\sigma(y  ,-v  ,-v' )=\sigma(y  ,v ,v' )$. An additional sufficient relation that ensures the no-drift condition is the following symmetry relation:
\begin{equation}
  \label{eq:tildesigma}
  \tilde\sigma(y,v',v) := \dfrac{\sigma(y,v',v)}{\psi(y,v')\psi^*(v)} = \tilde\sigma(y,v,v').
\end{equation} }

Finally, we observe that under \eqref{eq:symsigma}, $\psi^\eta(y,-v)$ and $\psi^\eta(-y,v)$ are also solutions of \eqref{eq:psieta}. Once properly normalized, since $\psi^*(-v)=\psi^*(v)$, we deduce that $\psi^\eta(y,v)=\psi^\eta(y,-v)=\psi^\eta(-y,v)$.


\medskip

Let us collect our main {\bf Assumptions}:
\begin{itemize}

\item[(H1)] The velocity variables $v$ lies in a compact, symmetric 
set $V$ of $\Rm^d$ equiped with a symmetric probability 
measure  $\nu$.
 There exist constants $C$, $\gamma>0$ such that
$\nu(\{v\in V,~|v\cdot\xi|\leq h\})\leq Ch^\gamma,\quad\mbox{for all }\xi\in S^{d-1}, \; h>0~.$
In particular, $\nu(\{v \in V,~|v\cdot\xi|=0\})=0$ for all $\xi\ne 0$ so that
$$
v\cdot\xi=0~~\hbox{ a.e. in } v \hbox{ implies } \xi =0~.
$$
\item[(H2)] The source term $S \in H^4(\Rm^d(C^\infty(Y),L^2(V)))$
\item[(H3)] The scattering coefficient $\sigma$ is in  ${\cal C}^0(V_{v} \times V_{v' };{\cal C}^\infty_{per}(\Rm^d_y))$. It is bounded from above and below by positive constants and is 1-periodic with respect to the variable $y$. The coefficient $\Sigma(y,v)$ is defined implicitly in \eqref{eq:psis}. 
\item[(H4)] The symmetry relations \eqref{eq:symsigma} hold so that the uniquely defined (after proper normalization) solutions $\psi^*(v)$, $\psi(y,v)$ and $\psi^\eta(y,v)$ of \eqref{eq:psi}, \eqref{eq:psis}, and \eqref{eq:psieta}, respectively, satisfy $\psi^*(-v)=\psi^*(v)$ and $\psi(y,v)=\psi(-y,v)=\psi(y,-v)$ as well as $\psi^\eta(y,v)=\psi^\eta(-y,v)=\psi^\eta(y,-v)$.

\end{itemize}

We want to stress again that the limit $\eta\to0$ is singular. The limit of $\psi^\eta$ as $\eta\to0$ is therefore not necessarily equal to $\psi(y,v)$.
We are now ready to state our main results on the Hilbert expansion of the solution $f^{\eps,\eta}:=f^\eta_\alpha$ of \eqref{eq:trepsalpha}. The main result of this paper is the following theorem, which provides a strong convergence result for the corrector to homogenization theory:
\begin{theorem}\label{th:main}  
Let $f^{\eps,\eta}$ be the solution of (\ref{eq:trepsalpha}). Then the following expansion holds
\begin{equation}\label{eq:main}\begin{array}{rcl}
\hspace{-1.5cm}\dps\Big\|f^{\eps,\eta} &-&n^{0,0}(x)\rho^0(\frac{\eta x}{\eps})\psi(\frac{\eta x}{\eps},v) \\[3mm]&-&\eta \Big[n^{0,0}(x) Q^{-1}\big(-v\cdot\nabla_y(\rho^0(\frac{\eta x}{\eps})\psi(\frac{\eta x}{\eps},v))\big)+n^{0,1}(x)\rho^0(\frac{\eta x}{\eps})\psi(\frac{\eta x}{\eps},v)\Big]\\
\\\dps&-&\dfrac{\eps}{\eta}\Big[\theta^{-1}(y)\psi(\frac{\eta x}{\eps},v)\cdot\nabla_xn^{0,0}(x)+n^{1,-1}(x)\rho^0(\frac{\eta x}{\eps})\psi(\frac{\eta x}{\eps},v)\Big]\Big\|_{L^2(\Rm^d\times V)}=o(\eta+\frac{\eps}{\eta}) \hspace{-2cm}
\end{array}
\end{equation}
where $\psi(y,v)$ and $\psi^*(v)$ are solutions of \eqref{eq:psi} and \eqref{eq:psis}, respectively, and $Q^{-1}$ is defined in \eqref{eq:Qinv}.
The microscopic density $\rho^0(y)$ is the unique solution of  the elliptic equation
$$
L(\rho^0)=0\quad \mbox{with the normalization}\quad \int_Y\rho^0(y)dy=1,
$$
where
$$
L(\rho)=-\int_V \psi^*(v) v\cdot\nabla_y(Q^{-1}(v\cdot\nabla_y(\psi(y,v)\rho(y))))d\nu(v).
$$
The  function $\theta^{-1}$ is given by
$$
\theta^{-1}=L^{-1}\left(\int_V(vQ^{-1}(-v\cdot\nabla_y\rho^0(y)\psi(y,v))-v\cdot\nabla_y(Q^{-1}(v\psi(y,v))\rho^0(y)))\psi^*(v)d\nu(v)\right),
$$
with $L^{-1}$ defined in Proposition \ref{propl} below on functions in $L^2(Y)$ that average to $0$ on $Y$
and  the macroscopic density is given by the diffusion equation
$$\begin{array}{rcl}
n^{0,0}(x)-\nabla_x\cdot({\bf D}\cdot\nabla_x n^{0,0}(x))&=&\dint_Y\int_V S(x,y,v)\psi^*(v)dyd\nu(v).
 \end{array}
$$
The diffusion coefficient in the preceding equation is given by the expression
$$
{\bf D}=\int_V\int_Y(\overline \chi^{*0}(y,v)\otimes v\rho^0(y)\psi(y,v)+\theta^{*-1}(y)\psi^*(v)\otimes vQ^{-1}(-v\nabla_y(\rho^0(y)\psi(y,v)))d\nu(v)dy
$$
in which we have defined 

$$\begin{array}{rcl}
\overline \chi^{*0}&=&Q^{*-1}\big(v\psi^*+v\cdot\nabla_y(\theta^{*-1}\psi^*)\big) \\[0mm]
\theta^{*-1}&=&L^{*-1}\Big(\dint_V \psi(y,v)v\cdot\nabla_y Q^{*-1}\big(v\psi^*(v)\big)d\nu(v)\Big).
\end{array}
$$
where $L^{*-1}$ is also defined in Proposition \ref{propl}.
The correctors $n^{0,1}(x)$ and $n^{1,-1}(x)$ satisfy the same elliptic equation as $n^{0,0}(x)$ with different source terms. Their expression is given explicitly in Proposition \ref{prop:densites} below.
\end{theorem}
Before proving this theorem, we make a few remarks.
\begin{remark}\rm
  The leading term in the expansion of $f^\eps_\alpha$ is given by $n^{0,0}(x)\rho^0(\frac{x}{\alpha})\psi(\frac{x}{\alpha},v)$. The behavior at the macroscopic level is given by $n^{0,0}(x)$, the solution of a standard diffusion equation. The microscopic level is given by the product of two terms. The first contribution to the product is the standard local solution $\psi(y,v)$, which indicates how particles are distributed in the $v$ variable for each $y\in Y$. The second, less standard, contribution is given by $\rho^0(\frac x \alpha)$ and indicates how the local (for each $y$) equilibria are related to one-another to generate a global equilibrium at the level of the cell $Y$.
\end{remark}
\begin{remark}\rm
\noindent The above expansion implies that when $\eta\ll\frac{\eps}{\eta}$, then the corrector is given by
$$
\theta^{-1}\psi(\frac{\eta x}{\eps},v)\cdot\nabla_xn^{0,0}(x)+n^{1,-1}(x)\rho^0(\frac{\eta x}{\eps})\psi(\frac{\eta x}{\eps},v).
$$
This is a regime of (relatively) low scattering where the corrector to homogenization (characterized by a term linear in $\nabla n^{0,0}$) dominates. The contribution $n^{1,-1}$ provides a correction to the influence of the local equilibria at each $y\in Y$ to a global equilibrium on $Y$. 

When $\frac{\eps}{\eta}\ll\eta$, the corrector is given instead by
$$
n^{0,0}(x) Q^{-1}\big(-v\cdot\nabla_y(\rho^0(\frac{\eta x}{\eps})\psi(\frac{\eta x}{\eps},v))\big)+n^{0,1}(x)\rho^0(\frac{\eta x}{\eps})\psi(\frac{\eta x}{\eps},v).
$$
This is the regime of high scattering, where the correction to approximating the transport solution by a diffusion approximation dominates the correction coming from the homogenization procedure. The passage from local to global equilibria on $Y$ generates a corrector described by $n^{0,1}(x)$.
\end{remark}
The rest of the paper is devoted to the proof of the theorem.


\section{A priori estimates}
\label{sec:apriori}
We start with an estimate that controls the remainder terms in the Hilbert expansion:
\begin{proposition} \label{prop:apriori} Let $\fea$ be the solution of (\ref{eq:trepsalpha}).  Then we have:
\begin{equation}
  \label{eq:ape2}
  \|\fea\| + \dfrac{1}{\eps} \|\fea-\pea\overline{\dfrac{\fea}{\pea}}\| \leq C \Big( \|\eps S_\alpha\| + 
     \| \overline{S_\alpha\peas} \|\Big),
\end{equation}
where $\|\cdot\|$ is the $L^2(\Rm^d\times V)$-norm  and for $u\in L^2(\Rm^d\times V)$ we have defined 
\begin{equation}
  \label{eq:baru}
  \bar u(x) = \dfrac{1}{|V|}\dint_V u(x,v) d\nu(v).
\end{equation}
\end{proposition}

\begin{proof}
We verify that 
\begin{equation}
  \label{eq:1pepes}
  (\eta\vny+Q^\eta) 1 =0 ,\qquad  (-\eta \vny + Q^{\eta*}) (\pe\pes)=0,\qquad Y\times V,
\end{equation}
where we have defined the rescaled transport operator
\begin{equation}
  \label{eq:Qeta}
  Q^\eta u(y,v) = \dint_V \dfrac{\sigma(y,v',v)\pe(y,v')}{\pe(y,v)} \big[u(y,v)-u(y,v')\big] d\nu(v').
\end{equation}

The reason for introducing the operator $Q^\eta$ is that 
\begin{equation}
  \label{eq:u}
  (\vnx + \dfrac1\eps Q^\eta_\alpha + \eps ) u = \eps F_\alpha(x,v) :=\eps F(x,\frac x\alpha,v),
\end{equation}
where we have defined
\begin{equation}
  \label{eq:uF}
  u(x,v) = \dfrac{ \fea(x,v)}{\pe(\frac x\alpha,v)},\qquad
  F(x,y,v) = \dfrac{ S(x,y,v)}{\pe(y,v)},
\end{equation}
and
\begin{equation}
  \label{eq:Qea}
  Q^\eta_\alpha u = \dint_V \dfrac{\sigma(\frac x\alpha,v',v)\pe(\frac x\alpha,v')}{\pe(\frac x\alpha,v)} \big[u(x,v)-u(x,v')\big] d\nu(v').
\end{equation}

Define the operator
\begin{equation}
  \label{eq:Tea}
  \Tea = \vnx + \dfrac1\eps \Qea.
\end{equation}
Then we recast \eqref{eq:u} as $(\eps+\Tea)u=\eps F_\alpha$.
We verify that 
\begin{equation}
  \label{eq:bQea}\begin{array}{rcl}
  \bQea(h) &:=& (\Tea h,\pea\peas h)\\&= &\dfrac1\eps\dint
    \sigma(\frac x\alpha,v\rq{},v)\pea(x,v\rq{})\peas(x,v)\frac{|h(x,v)-h(x,v\rq{})|^2}2 dxd\nu(v)d\nu(v)\rq{}.
    \end{array}
\end{equation}
This comes from the fact that $\vnx(\pea\peas)=\frac1\eps Q^{\eta*}_\alpha(\pea\peas)$ so that
\begin{displaymath} 
 \bQea(h)  = \dfrac1\eps(\Qea(h) h - \Qea\frac {h^2}2,\pea\peas).
 \end{displaymath}

Let $\rho=\rho(x)$ and $s=u-\rho$ for an arbitrary $\rho(x)$ independent of $v$. Then we find that $\bQea(\rho)=0$ and more importantly that $\bQea(u)=\bQea(s)$.  Now for $s$ such that $\int_V s d\nu(v)=0$, we find that 
\begin{equation}
  \label{eq:coerc}
  \bQea(s) \geq \dfrac\beta\eps\|s\|^2,
\end{equation}
for some $\beta>0$ as is clear from \eqref{eq:bQea} provided that $\pea\peas$ is bounded from below by a positive constant uniformly in $\eta$. We thus define $\rho=\bar u$ and $s=u-\bar u$ so that $\bar u$ is independent of $v$ and $s$ mean zero in $v$. Multiplying \eqref{eq:u} by $u\pea\peas$ and integrating yields
\begin{equation}
  \label{eq:varequ}
  \eps(u,u\pea\peas) + \bQea(s) = (\eps F_\alpha , u \pea\peas).
\end{equation}
As a consequence, 
\begin{equation}
  \label{eq:firstineq}
  \eps \|u\|^2 + \dfrac{1}{\eps} \|s\|^2 \leq |(\eps F_\alpha , s \pea\peas)|
       + |(\eps \overline{F_\alpha\pea\peas}, \rho)|.
\end{equation}
From this, we deduce the a priori estimate
\begin{equation}
  \label{eq:ape}
  \|u\| + \dfrac{1}{\eps} \|u-\bar u\| \leq C \Big( \|\eps F_\alpha\| + 
     \| \overline{F_\alpha\pea\peas} \|\Big).
\end{equation}
In the variables $\fea$, this is equivalent to (\ref{eq:ape2}).
\end{proof}

\section{Expansion in $\eps$}
\label{sec:eps}
To emphasize the dependency in $\eta$, let us denote $f^{\eps,\eta}:=f^\eps_\alpha$ the solution of (\ref{eq:trepsalpha}). As in the standard derivation of diffusion approximations, we first expand $f^{\eps,\eta}$ in powers of $\eps$ at a fixed (arbitrary) value of $\eta$. We prove the following proposition.

\begin{proposition}\label{prop:exp} The solution $f^{\eps,\eta}$ can be expanded as follows 
\begin{equation}\label{eq:expeps}
f^{\eps,\eta}= n^{0,\eta}(x)\psi^\eta(\frac{\eta x}{\eps},v)+\eps f^{1,\eta}(x,\frac{\eta x}{\eps},v)+\eps^2f^{2,\eta}(x,\frac{\eta x}{\eps},v)+\eps^3f^{3,\eta}(x,\frac{\eta x}{\eps},v)+r^{\eps,\eta}(x,v)
\end{equation}
where we have defined:
\begin{eqnarray*}
f^{1,\eta}&=&\dps T^{\eta -1}(-v\cdot\nabla_x(n^{0,\eta}\psi^\eta))+n^{1,\eta}\psi^\eta\\
f^{2,\eta}&=&\dps T^{\eta -1}(-v\cdot\nabla_xn^{1,\eta}\psi^\eta)+\overline f^{2,\eta}\\
\overline f^{2,\eta}&=&\dps T^{\eta -1}(S-v\cdot\nabla_xT^{\eta -1}(n^{0,\eta}\psi^\eta) -f^{0,\eta})\\
f^{3,\eta}&=&\dps T^{\eta -1}(-v\cdot\nabla_x f^{2,\eta}-f^{1,\eta}).
\end{eqnarray*}
The operator $T^{\eta-1}$ is defined in \eqref{eq:Tinv}.
The density $n^{0,\eta}$ satisfies the diffusion equation
\begin{equation}\label{eq:difneta}
n^{0,\eta}-\nabla_x\cdot\int_{Y\times V}\hspace{-.3cm} v\chi^\eta (y,v)\psi^{*}(v)d\nu(v)dy\,\nabla_xn^{0,\eta}=\int_{Y\times V}\hspace{-.3cm} S(x,y,v)\psi^{\eta*}(v)d\nu(v)dy
\end{equation}
where $\chi^\eta=T^{\eta -1}(v\psi^\eta)$ and the density for the correcting term $n^{1,\eta}$ satisfies
\begin{equation}\label{eq:difn1eta}
n^{1,\eta}-\nabla_x\cdot\int_{Y\times V}\hspace{-.6cm} v\chi^\eta(y,v)\psi^{*}(v)d\nu(v)dy\,\nabla_xn^{1,\eta}=-\int_{Y\times V}\hspace{-.6cm} v\cdot\nabla_x\overline f^{2,\eta}(x,y,v)\psi^*(v)d\nu(v)dy.
\end{equation}
The remainder term $r^{\eps,\eta}$ satisfies the following estimate:
\begin{equation}\label{eq:estr1} \begin{array}{rcl}
||r^{\eps,\eta}||_{L^2_{xv}}&\leq& \eps^2||\overline{v\cdot\nabla_xf^{3,\eta}\psi^{\eta*}}||_{L^2_xL^\infty_yL^2_v}+ \eps^3||v\cdot\nabla_xf^{3,\eta}||_{L^2_xL^\infty_yL^2_v}\\[2mm] && \quad +\eps^2||f^{2,\eta}||_{L^2_xL^\infty_yL^2_v}+\eps^3||f^{3,\eta}||_{L^2_xL^\infty_yL^2_v}. \end{array}
\end{equation}
\end{proposition}
\begin{proof}
We propose the following ansatz for $f^{\eps,\eta}$ :
$$
f^{\eps,\eta}= f^{0,\eta}(x,\frac{\eta x}{\eps},v)+\eps f^{1,\eta}(x,\frac{\eta x}{\eps},v)+\eps^2 f^{2,\eta}(x,\frac{\eta x}{\eps},v)+\eps^3 f^{3,\eta}(x,\frac{\eta x}{\eps},v)+r^{\eps,\eta}(x,v).
$$
We plug the ansatz into the transport equation and equate like powers of $\eps$ to obtain the following sequence of equations. At the leading order, we obtain
$$
T^\eta f^{0,\eta}=0\quad \mbox{which yields that }\quad f^{0,\eta}=n^{0,\eta}(x)\psi^\eta(y,v).
$$
We recall that $T^\eta=\eta v\cdot\nabla_y+Q$. At the next order, $f^{1,\eta}$ satisfies
$$
T^\eta f^{1,\eta}=-v\cdot\nabla_xf^{0,\eta}(x,y,v)=-\psi^\eta(y,v) v\cdot\nabla_xn^{0,\eta}(x).
$$
We can then rewrite $f^{1,\eta}$ as
$$
f^{1,\eta}(x,y,v)=- \chi^\eta(y,v)\cdot\nabla_xn^{0,\eta}(x) + n^{1,\eta}(x)\psi^\eta(y,v)
$$
by defining $\chi^\eta=T^{\eta,-1}(v\psi^\eta)$, the unique solution to 
$$
T^\eta \chi^\eta=v\psi^\eta\quad \mbox{such that}\quad \int_V\int_Y\chi^\eta (y,v)\psi^*(v)dyd\nu(v)=0.
$$
We can apply the inverse operator $T^{\eta,-1}$ defined in \eqref{eq:Tinv} to the source term $v\psi^\eta$ because the following standard {\em no-drift} condition is satisfied:
$$
\int_V\int_Y v\psi^\eta(y,v)\psi^{*}(v)d\nu(v)dy=0,
$$
thanks to assumption (H4).

The second-order equation reads 
$$
T^\eta f^{2,\eta}=-f^{0,\eta}(x,y,v)+S(x,y,v)-v\cdot\nabla_xf^{1,\eta}(x,y,v)
$$
and the compatibility equation that the right-hand side must satisfy to be in the domain of definition of $T^{\eta,-1}$ gives the diffusion equation (\ref{eq:difneta}) for $n^{0,\eta}$.

Introduce $\overline f^{2,\eta}$ the solution to 
$$
T^\eta \overline f^{2,\eta}=-f^{0,\eta}(x,y,v)+S(x,y,v)-v\cdot\nabla_x(- \chi^\eta(y,v)\cdot\nabla_xn^{0,\eta}(x)).
$$
The third-order  equation is the following
$$
T^\eta f^{3,\eta}=-v\cdot\nabla_xf^{2,\eta}(x,y,v)-f^{1,\eta}(x,y,v).
$$
and the corresponding compatibility equation gives a diffusion equation for $n^{1,\eta}$
$$
n^{1,\eta}(x)-\nabla_x\cdot\int_{Y\times V}\hspace{-.3cm} v\chi^\eta(y,v)\psi^*(v)d\nu(v)dy\nabla_xn^{1,\eta}(x)=-\int_{Y\times V}\hspace{-.3cm}v\cdot\nabla_x\overline f^{2,\eta}(x,y,v)\psi^*(v)d\nu(v)dy
$$
since
$$
\int_Y\int_V\psi^*(v) \chi^\eta(y,v)\cdot\nabla_xn^{0,\eta}(x)=0.
$$
With the above expressions, we verify that remainder term now satisfies that:
$$
r^{\eps,\eta}+\frac{1}{\eps}v\cdot\nabla_xr^{\eps,\eta}+\frac{1}{\eps^2}Q(r^{\eps,\eta})=-\eps^2f^{2,\eta}(x,\frac{\eta x}{\eps},v)-\eps^3f^{3,\eta}(x,\frac{\eta x}{\eps},v)-\eps^2 v\cdot\nabla_xf^{3,\eta}(x,\frac{\eta x}{\eps},v).
$$
Thanks to the estimate (\ref{eq:ape2}) proved in the previous section, we obtain (\ref{eq:estr1}). 
\end{proof}
In order to analyze the behavior of the terms $f^{0,\eta}$, $f^{1,\eta}$, $f^{2,\eta}$ and $f^{3,\eta}$ as $\eta\to0$, we need to study some auxiliary equations. Such studies are conducted in the following section.
\section{Expansion of the auxiliary functions}
\label{sec:aux}
This section is devoted to the asymptotic expansion as $\eta\to0$ of the three cell functions $\psi^\eta$, $\chi^\eta$, and $\chi^{\eta *}$.

\subsection{Expansion of $\psi^\eta$}

Recall that $\psi^{\eta}$ satisfies 
$$
\eta v\cdot\nabla_y \psi^\eta +Q(\psi^\eta)=0\quad \mbox{and}\quad \int_Y\int_V \psi^\eta(y,v)\psi^*(v)d\nu(v)dy=1.
$$
We prove the following expansion

\begin{proposition} The function  $\psi^\eta$, solution to (\ref{eq:psieta}) satisfies
\begin{equation}\label{eq:exp-psi-eta}
\psi^\eta(y,v)=\psi^0(y,v)+\eta\psi^1(y,v)+\eta^2\psi^2(y,v)+\tilde r^\eta_{p}(y,v),\quad||\tilde r^\eta_{p}||_{H^k(Y,L^2( V))}\leq C\eta^3,
\end{equation}
with
\begin{equation}\label{eq:psis1}
\begin{array}{rcll}
\psi^0(y,v)&=&\rho^0(y)\psi(y,v),\\
 \psi^1(y,v)&=&Q^{-1}(-v\cdot\nabla_y\psi^0(y,v))\\
\psi^2(y,v)&=&Q^{-1}(-v\cdot\nabla_y\psi^1(y,v))+\rho^2(y)\psi(y),\end{array}
\end{equation}
where
\begin{eqnarray}
&& L(\rho^0)=0 \mbox{ with }\int \rho^0(y)dy=1\\
&& \dps \rho^2= -L^{-1}\left(\int_V\psi^*(v)v\cdot\nabla_y(Q^{-1}(-v\cdot\nabla_y(Q^{-1}(-v\cdot\nabla_y\psi^1(y,v)))))d\nu(v)\right).
\end{eqnarray}
The operator $L$ and its inverse defined on functions in $L^2(Y)$ with vanishing average over $Y$ are defined in Proposition \ref{propl} in the appendix.
From this expansion, we deduce that:
$$
||\psi^\eta||_{L^2_vL^\infty_y}\leq C\quad\mbox{and}\quad ||\int_Vv\psi^\eta\psi^* d\nu(v)||_{L^\infty_y}\leq C\eta.
$$
Moreover, we find that $\psi^0$ and $\psi^2$ are even functions and $\psi^1$ is an odd function in the variable $v$.
\end{proposition}

\begin{proof}
We plug the ansatz into the equation for $\psi^\eta$ and equate like powers of $\eta$. The resulting series of equations is the following. First we have
$$
Q(\psi^0)=0,
$$
which implies that $\psi^0=\rho^0(y)\psi(y,v)$. Next, we obtain 
$$
Q(\psi^1)=-v\cdot\nabla_y(\rho^0(y)\psi(y,v)).
$$
Since thanks to assumption {(H4)} the no-drift condition $\int_V v\psi(y,v)\psi^*(v) d\nu(v)=0$ is satisfied, the compatibility condition for the previous equation is fulfilled and we can solve the equation for $\psi^1$ to obtain
$$
\psi^1=Q^{-1}(-v\cdot\nabla_y(\rho^0(y)\psi(y,v))).
$$

The compatibility condition for the third equation gives the equation for $\rho^0$
since for
$$
Q(\psi^2)=-v\cdot\nabla_y (\psi^1(y,v)) \quad\mbox{i.e.,}\quad \psi^2=Q^{-1}(-v\cdot\nabla_y\psi^1(y,v))+\rho^2(y)\psi(y,v)
$$
to have a solution, we need
$$
\int_V \psi^*(v)v\cdot\nabla_y (\psi^1(y,v))d\nu(v)=0.
$$
This relation is satisfied if 
$$
L(\rho^0)=0\quad \mbox{with}\quad L(\rho)=-\int_V \psi^*(v) v\cdot\nabla_y(Q^{-1}(v\cdot\nabla_y(\psi(y,v)\rho(y))))d\nu(v).
$$
The precise study of $L$ and its adjoint $L^*$ defined by
$$
L^*(\rho)=-\int_V \psi(y,v) v\cdot\nabla_y({Q^*}^{-1}(v\cdot\nabla_y(\psi^*(v)\rho))d\nu(v)
$$
is conducted in the appendix.
The next equation is
$$
Q(\psi^3)=-v\cdot\nabla_y(\psi^2).
$$
The corresponding compatibility equation is satisfied since $\psi^2$ is an even function of $v$.
Then we write
$$
Q(\psi^4)=-v\cdot\nabla_y(\psi^3).
$$
Its compatibility condition  gives the following equation for $\rho^2$:
\begin{equation}\label{eq:ro2}
L(\rho^2)=-\int_V\psi^*(v)v\cdot\nabla_y(Q^{-1}(-v\cdot\nabla_y(Q^{-1}(-v\cdot\nabla_y\psi^1(y,v)))))d\nu(v).
\end{equation}
Recall that $L^*(1)=0$. By the Fredholm alternative, equation (\ref{eq:ro2}) thus admits a unique mean zero solution since
$$
\int_Y\int_V\psi^*(v)v\cdot\nabla_y(Q^{-1}(-v\cdot\nabla_y(Q^{-1}(-v\cdot\nabla_y\psi^1(y,v)))))d\nu(v)dy=0.
$$

The remainder term $r_{p}^\eta$ satisfies
$$
\eta v\cdot\nabla r_p^\eta+Q(r_p^\eta)=-\eta^{5}v\cdot\nabla_y(\psi^4)
$$
and if we impose 
$
\int_y \rho^0(y)dy=1$  and $\int _Y \rho^2(y)dy=0,
$
we get
$$
\int_Y\int_V  r_{p}^\eta(y,v) \psi^*(v)d\nu(v)dy= 0.
$$
Then using the technical regularity estimate (\ref{est-t}) presented and proved in the appendix below, we obtain
$$
||r_{p}^\eta||_{H^k(Y),L^2( V))}\leq  \eta^3||\psi^4||_{H^{k+1}(Y,L^2(V))}.
$$
The remainder involved in the proposition is given by
$
\tilde r^\eta_p=\eta^3\psi^3+\eta^4\psi^4+r^\eta_p
$
and then satisfies
$$
||\tilde r_{p}^\eta||_{H^k(Y),L^2( V))}\leq  \eta^3.
$$
By taking $k$ large enough, we get the desired $L^\infty_y$ bounds.
\end{proof}
\subsection{Expansion of $\chi^\eta$}

We want to expand $\chi^\eta$ satisfying
$$
\eta v\cdot\nabla_y \chi^\eta+Q(\chi^\eta)=v\psi^\eta\quad\mbox{ and }\quad \int_Y\int_V  \chi^\eta(y,v)\psi^*(v) d\nu(v)dy=0.
$$

\begin{proposition}
We prove the following expansion
$$
\chi^\eta(y,v)= \frac{1}{\eta}\theta^{-1}(y)\psi(y,v)+ Q^{-1}(v\psi)\rho^0(y)-Q^{-1}(v\cdot\nabla_y(\theta^{-1}\psi(y,v)))+\tilde r^\eta_{\chi}(y,v),\quad
$$
with $||\tilde r^\eta_{\chi}||_{H^k(Y,L^2(V))}\leq C\eta$,
where we have
\begin{eqnarray*}
&&L(\rho^0)=0\quad \mbox{with}\quad \int_Y\rho^0(y)dy=1\\
&& L\theta^{-1}=\int_V(vQ^{-1}(-v\cdot\nabla_y\rho^0(y)\psi(y,v))\psi^*(v)-v\psi^*(v)\cdot\nabla_y(Q^{-1}(v\psi)(y,v)\rho^0(y)))d\nu(v).
\end{eqnarray*}
Moreover
\begin{equation}
\label{eq:intchi}
||\chi^\eta||_{L^2_vL^\infty_y}\leq \frac{C}{\eta}\quad\mbox{and}\quad||\int_Vv\chi^{\eta}(y,v)\psi^*(v)d\nu(v)||_{L^\infty_y}\leq C.
\end{equation}
\end{proposition}
\begin{proof}In view of estimate (\ref{est-t-2}) obtained in the appendix below, we expect the expansion of $\chi^\eta$ to start at the order $\dps\frac{1}{\eta}$. Thus, we begin with
$$
Q(\chi^{-1})=0\quad\mbox{or}\quad\chi^{-1}(y,v)=\theta^{-1}(y)\psi(y,v),
$$
Then at the order $O(\eta^0)$ we get
$$
Q(\chi^0)=v\psi^0-v\cdot\nabla_y\chi^{-1}.
$$
Thanks to the no-drift condition and (H4), the compatibility condition for this equation  is satisfied.
Next, we write
$$
Q(\chi^1)= v\psi^1-v\cdot\nabla_y\chi^0\quad \mbox{ i.e., }\quad \chi^1=Q^{-1}(v\psi^1-v\cdot\nabla_y\chi^0)(y,v)+\theta^1(y)\psi(y,v).
$$
The compatibility condition for the equation giving $\chi^1$ gives the expression of $\theta^{-1}$ since it satisfies the following elliptic equation
$$
L(\theta^{-1})=\int_V(v\psi^1(y,v)\psi^*(v)-v\psi^*(v)\cdot\nabla_y(\chi(y,v)\rho^0(y)))d\nu(v).
$$
This equation has a solution since
$$
\int_Y\int_V(v\psi^1(y,v)\psi^*(v)-v\psi^*(v)\cdot\nabla_y(\chi(y,v)\rho^0(y)))d\nu(v)dy=0
$$
because $\psi^1$ defined in \eqref{eq:psis1} is odd in $y$ as a gradient of an even function and the second contribution can be written as a divergence in $y$ and hence averages to $0$ over the cell $Y$.
Note that $\theta^{-1}$ is odd in the $y$ variable and that therefore 
$$\int_Y\int_V  \chi^{-1}(y,v)\psi^*(v)d\nu(v)dy=0.$$

Finally, we write
$$
Q(\chi^2)=v\psi^2-v\cdot\nabla_y\chi^1
$$
for which the compatibility condition holds since $\psi^2$ is even with respect to $v$ and $\chi^1$ is even. Indeed, $\psi^1$ is odd with repect to $v$   (even with respect to $(y,v)$) and hence so is $\chi^0$.

The remainder term satisfies
$$
\eta v\cdot \nabla_y r_{\chi}^\eta+Q(r_\chi^\eta)=- \eta^3v\cdot\nabla_y\chi^2+v r_p^\eta .
$$
We chose $\theta^1$ such that $\int_V(-\eta^3v\cdot\nabla_y\chi^2(y,v)+v r_p^\eta (y,v))\psi^*(v)d\nu(v)=0$ for any $y \in Y$. Note that this is possible since
$
\int_V\int_Y(-\eta^3v\cdot\nabla_y\chi^2+v r_p^\eta) \psi^*d\nu(v) dy=0
$
because $r^\eta_p$  is even with respect to  $(y,v)$ whereas $\chi^2$ is odd.
The fact that $\chi^2$ and $vr_p^\eta$ are odd with respect to $(y,v)$ implies, thanks to Proposition \ref{propl}, that $\theta^1$ is odd and  therefore $\int_Y\int_V  \chi^{1}(y,v)\psi^*(v)d\nu(v)dy=0$,
which finally gives
$$
\int_Y\int_V r_\chi^{\eta}(y,v)\psi^*(v)d\nu(v)dy=0.
$$
By using estimate (\ref{est-t}) and since $||- \eta^3v\cdot\nabla_y\chi^2+vr_p^\eta ||_{H^k(Y,L^2(V))}\leq C\eta^2$, we have that
$$
||r_\chi^{\eta}||_{H^k(Y,L^2(V))}\leq C\eta.
$$
This concludes the proof of the result.
\end{proof}

\subsection{Expansion of $\chi^{\eta*}$}
\begin{proposition}Let $\chi^{\eta *}$ be the solution of 
$$
-\eta v\cdot\nabla_y\chi^{\eta*}+Q^*(\chi^{\eta*})=v\psi^*.
$$
The following expansion holds
$$\chi^{\eta*}=\frac{1}{\eta} \theta^{*-1}(y)\psi^*(v)+\chi^{*0}(y,v)+\eta\chi^{*1}(y,v)+\tilde r^\eta_{\chi^*}(y,v)
$$
with $\dps ||\tilde r^\eta_{\chi^*}||_{H^k(Y,L^2(V))}\leq C\eta^2$, where
\begin{eqnarray}
\chi^{*0}(y,v)&=&Q^{*-1}(v\psi^*+v\cdot\nabla_y\theta^{*-1}\psi^*)(y,v)+\theta^{*0}(y)\psi^*(v)\\
L^*(\theta^{*-1})&=&\int_V \psi(y,v)v\cdot\nabla_y Q^{*-1}(v\psi^*(v))d\nu(v),\\
L^*(\theta^{*0})&=& \int _V \psi(y,v)v\cdot\nabla_yQ^{*-1}(v\cdot\nabla_yQ^{*-1}(v\psi^*+v\cdot\nabla_y\chi^{*-1}))d\nu(v),\\
\chi^{*1}(y,v)&=&Q^{*-1}(v\cdot\nabla_y\chi^{*0}(y,v))+\theta^{*1}(y)\psi^*(y,v),\\
L^*(\theta^{*1})&=& \int _V \psi(y,v)v\cdot\nabla_yQ^{*-1}(v\cdot\nabla_yQ^{*-1}(\chi^{*0}))d\nu(v).
\end{eqnarray}

\end{proposition}

\begin{proof}
The proof is a straightforward adaptation of the proof of the same result in \cite{NBAPuVo}. We briefly sketch it to be self-contained.
As in the expansion of $\chi^\eta$, we start at the order $\frac{1}{\eta}$ and write
$$
\chi^{\eta*}=\frac{1}{\eta}\chi^{*-1}+\chi^{*0}+\eta \chi^{*1}+\eta^2 \chi^{*2}+\eta^3\chi^{*3}+r_{\chi^*}^\eta
$$
with
$$
Q^*(\chi^{*-1})=0\mbox{ so that }\chi^{*-1}(y,v)=\theta^{*-1}(y)\psi^*(v).
$$
Then we have
$$
Q^*(\chi^{*0})=v\psi^*+v\cdot\nabla_y\chi^{*-1}
$$
which admits, thanks to the no-drift condition, a solution  given by
$$
\chi^{*0}(y,v)=Q^{*-1}(v\psi^*+v\cdot\nabla_y\chi^{*-1})(y,v)+\theta^{*0}(y)\psi^*(v).
$$
The first order equation is
$$
Q^*(\chi^{*1})=v\cdot\nabla_y\chi^{*0}.
$$
The compatibility equation for this equation gives
$$
L^*(\theta^{*-1}):=-\int_V \psi(y,v)v\cdot\nabla_yQ^{*-1}(v\cdot\nabla_y(\psi^*(v)\theta^{*-1})d\nu(v)=\int_V \psi(y,v)v\cdot\nabla_y Q^{*-1}(v\psi^*(v))d\nu(v).
$$
This equation has a solution since
$$
\int_Y\rho^0(y)\int_V  \psi(y,v)v\cdot\nabla_y Q^{*-1}(v\psi^*(v))d\nu(v)dy=0
$$
because $\rho^0\psi$ is even with respect to the variable $(y,v)$ whereas $v\psi^*$ is odd.
Then $\chi^{*1}$ can be written  as
$$
\chi^{*1}(y,v)=Q^{*-1}(v\cdot\nabla_y\chi^{*0}(y,v))+\theta^{*1}(y)\psi^*(y,v).
$$
The compatibility condition for 
$$
Q^*(\chi^{*2})=v\cdot \nabla_y\chi^{*1}
$$
gives the equation determining $\theta^{*0}$ since we need
$$
L^*(\theta^{*0})= \int _V \psi(y,v)v\cdot\nabla_yQ^{*-1}(v\cdot\nabla_yQ^{*-1}(v\psi^*+v\cdot\nabla_y\chi^{*-1}))d\nu(v).
$$
Note that the parity properties ensure that the previous equation has a solution.
The final equation is
$$
Q^*(\chi^{*3})=v\cdot \nabla_y\chi^{*2}
$$
whose compatibility condition gives the equation for $\theta^{*1}$
$$
L^*(\theta^{*1})= \int _V \psi(y,v)v\cdot\nabla_yQ^{*-1}(v\cdot\nabla_yQ^{*-1}(\chi^{*0}))d\nu(v).
$$
The remainder term satisfies
$$
\eta v\cdot \nabla_y r_{\chi^*}^\eta-Q^*(r_{\chi^*}^\eta)=\eta^4v\cdot\nabla\chi^{*3}.
$$
As for $\chi^\eta$, $\theta^{*i}$ are odd for $i=-1,0,1$ and then, $\int_V\int_Yr_{\chi^*}^\eta\psi^*d\nu(v)dy=0$, and then, thanks to  estimate (\ref{est-t*}), we get
$$
|| r_{\chi^*}^\eta||_{H^k(Y,L^2(V))}\leq C\eta^2
$$
which concludes the proof.
\end{proof}


\section{Estimates on the densities}
\label{sec:density}

We are now ready to analyze the expansion in $\eta$ of the densities $n^{k,\eta}$ for $k=0,1$ obtained in Proposition \ref{prop:exp}. 
We start with the following simple lemma:
\begin{lemma}
\label{lem:densities}

Let $n^\eta$ be the solution of 
$$
n^\eta(x)-\nabla_x\cdot( D^\eta \cdot\nabla_x n^\eta(x))=S^\eta(x)
$$
with 
\begin{eqnarray*}
D^\eta&=&D^0+\eta D^1+\eta \tilde D^\eta, \quad\mbox{where}\quad ||| \tilde D^\eta|||\rightarrow_{\eta\rightarrow 0} 0,\quad |||\cdot||| \mbox{ a norm on }d\times d\mbox{ matrices}\\
S^\eta&=&S^0+\eta S^1+\eta \tilde S^\eta \quad\mbox{where}\quad S^0\in H^k(\Rm^d), \,S^1\in H^k(\Rm^d) \mbox{ and } ||\tilde S^\eta||_{H^k(\Rm^d)}\rightarrow_{\eta\rightarrow 0} 0.\end{eqnarray*} 
We decompose
$$
n^\eta=n^0+\eta n^1+\eta \tilde n^\eta
$$
with
$$
n^0(x)-\nabla_x\cdot({D^0}\nabla_x n^0(x))=S^0(x),\quad ||n^0||_{H^{k+1}(\Rm^d)}\leq C||S^0||_{H^k(\Rm^d)}
$$
and
$$
n^1(x)-\nabla_x\cdot({D^0}\nabla_x n^1(x))=S^1(x)+\nabla_x\cdot (D^1\cdot\nabla_x n^0(x)), \,\, ||n^1||_{H^{k+1}(\Rm^d)}\leq C||S^0||_{H^k(\Rm^d)}+||S^1||_{H^k(\Rm^d)}.
$$
Then $\tilde n^\eta$ satisfies
$$||\tilde n^\eta||_{H^{k+1}(\Rm^d)}\rightarrow 0\quad\mbox{when}\quad \eta\rightarrow 0.$$
\end{lemma}
\begin{proof}
Indeed $\tilde n^\eta$ satisfies the equation
$$
\tilde n^\eta-\nabla_x\cdot({D^0}\nabla_x\tilde n^\eta)-\nabla_x\cdot[(\eta D^1+\eta\tilde D^\eta)\nabla_x\tilde n^\eta]=\tilde S^\eta+\nabla_x\cdot[\tilde D^\eta \nabla_xn^0+(\eta D^1+\eta\tilde D^\eta)\nabla_xn^1].
$$
A standard a priori estimate shows that  $||\tilde n^\eta||_{H^1(\Rm^d)}\rightarrow 0$ when $\eta\rightarrow 0$. Upon differentiating the equation $k$ times, we obtain the result in $H^{k+1}(\Rm^d)$.
%
\end{proof}

We apply this lemma to derive a convergence result for $n^{0,\eta}$ and $n^{1,\eta}$.
\begin{proposition}
\label{prop:densites}
Assume $S\in H^k(\Rm^d)$. First, we have that
$$||n^{0,\eta}-n^{0,0}-\eta n^{0,1}||_{H^{k+1}(\Rm^d)}\rightarrow 0$$ where $n^{0,0}\in H^{k+1}(\Rm^d) $ is the solution to
$$
n^{0,0}-\nabla_x\cdot({\bf D}\cdot\nabla_x n^{0,0})=\int_Y\int_V S(x,y,v)\psi^*(v)dyd\nu(v)
$$
and $n^{0,1}\in H^{k+1}(\Rm^d) $ is the solution to
$$
n^{0,1}-\nabla_x\cdot({\bf D}\cdot\nabla_x n^{0,1})=\nabla_x\cdot({\bf D^1}\cdot\nabla_x n^{0,0}),
$$
$$\begin{array}{rcl} {\bf D}&=&\dint_V\int_Y(\chi^{*0}(y,v)\otimes v\psi^0(y,v)+\chi^{*-1}(y,v)\otimes v\psi^1(y,v))d\nu(v)dy \\ [3mm]
&=&\dint_V\int_YQ^{*-1}(v\psi^*+v\cdot\nabla_y\theta^{*-1}\psi^*)(y,v)\otimes v\rho^0(y)\psi(y,v)\\
&&\phantom{\dint_V\int_YQ^{*-1}(-v\psi^*+v\cdot\nabla_y}+\theta^{*-1}(y)\psi^*(v)\otimes vQ^{-1}(-v\cdot\nabla_y(\rho^0(y)\psi(y,v)))d\nu(v)dy\end{array}
$$
$$\begin{array}{rcl}
{\bf D^1}&=&\dint_V\int_Y(\chi^{*-1}(y,v)\otimes v\psi^2(y,v)+\chi^{*0}(y,v)\otimes v\psi^1(y,v)+\chi^{*1}(y,v)\otimes v\psi^0(y,v))d\nu(v)dy \\ [3mm]
&=&\dint_V\int_Y(\theta^{*-1}(y)\psi^*(v)\otimes v Q^{-1}(-v\cdot\nabla_y(Q^{-1}(-v\cdot\nabla_y(\rho^0(y)\psi(y,v)))))d\nu(v)dy\\ [3mm]
&&\dint_V\int_Y[Q^{*-1}(v\psi^*+v\cdot\nabla_y\theta^{*-1}\psi^*)+\theta^{*0}\psi^*]\otimes v Q^{-1}(-v\cdot\nabla_y(\rho^0(y)\psi(y,v))d\nu(v)dy\\ [3mm]
&&\dint_V\int_Y Q^{*-1}(v\cdot\nabla_yQ^{*-1}(v\psi^*+v\cdot\nabla_y\theta^{*-1}\psi^*)+\theta^{*1}(y)\psi^*(y,v))\otimes v\rho^0(y)\psi(y,v)d\nu(v)dy.
\end{array}
$$
Second, we have that 
\begin{equation}\label{eq:n1}
||n^{1,\eta}-\frac{1}{\eta}n^{1,-1}||_{H^{k-1} (\Rm^d)}\rightarrow_{\eta\rightarrow 0} 0
\end{equation}
with $n^{1,-1}\in H^{k-1} (\Rm^d)$ the solution of 
\begin{equation}
n^{1,-1}-\nabla_x\cdot({\bf D}\cdot\nabla_x n^{1,-1})=S^{1,-1}(x)
\end{equation}
where $S^{1,-1}\in H^{k-2}(\Rm^d)$ is given by
$$\begin{array}{rcl}
\hspace{-.25cm} S^{1,-1}(x)\!\!\!&=&\!\!\!\dps-\nabla_x\!\!\cdot\!\!\int_{Y\times V}\hspace{-.5cm} [-n^{0,0}(x)\psi(y,v)\!+\!S(x,y,v)\!+\!v\!\cdot\!\nabla_x(\chi^0(y,v)\!\cdot\!\nabla_xn^{0,0}(x))]\chi^{*-1}(y,v)d\nu(v)dy\\
&&\dps-\nabla_x\cdot\int_{Y\times V}\hspace{-.05cm}[-v\cdot\nabla_x(- \chi^{-1}(y,v)\cdot\nabla_xn^{0,0}(x))]\chi^{*0}(y,v)d\nu(v)dy .
\end{array}$$
\end{proposition}
\begin{proof}
Recall that $n^{0,\eta}$ satisfies
$$
n^{0,\eta}+\nabla_x\cdot(D^\eta\cdot\nabla_xn^{0,\eta})=\int_Y\int_V S(x,y,v)\psi^{*}(v)d\nu(v)dy.
$$
By using the definition of $\dps D^\eta=\int_V\int_Y \chi^{\eta*}(y,v)\otimes v\psi^\eta(v) dyd\nu(v)$ and the expansions of $\psi^\eta$ and $\chi^{\eta*}$ obtained in the preceding section, we obtain that
$$ 
D^\eta={\bf D}+ \eta {\bf D^1}+\eta \tilde D^\eta\quad \mbox{ with }\quad |||\tilde D^\eta||| \rightarrow_{\eta\rightarrow 0} 0.
$$
Then Lemma \ref{lem:densities} gives the first result.
On the another hand, $n^{1,\eta}$ satisfies
$$
n^{1,\eta}-\nabla_x\cdot(D^\eta\cdot\nabla_xn^{1,\eta})=S^{1,\eta}(x)$$
with 
\begin{eqnarray*}
\!S^{1,\eta}(x)\!\!\!&=&\!\!-\nabla_x\cdot\int_Y\int_V (\overline f^{2,\eta}(x,y,v))v\psi^{\eta*}(v)d\nu(v)dy\\
&=&\!\!-\nabla_x\int_Y\int_V \overline f^{2,\eta}T^{\eta *}\chi^{\eta*}(y,v)d\nu(v)dy
=-\nabla_x\int_Y\int_V T^{\eta }\overline f^{2,\eta}\chi^{\eta*}(y,v)d\nu(v)dy\\
&=&\!\!-\nabla_x\!\!\int_{Y\times V}\hspace{-.5cm}  [-n^{0,\eta}(x)\psi^\eta(y,v)\!+\!S(x,y,v)
\!+\!v\!\cdot\!\nabla_x(\chi^\eta(y,v)\!\cdot\!\nabla_xn^{0,\eta}(x))]\chi^{\eta*}(y,v)d\nu(v)dy\\
&=& \frac{1}{\eta} S^{1,-1}(x) + o(\frac{1}{\eta}).
\end{eqnarray*}
Since $S^{1,\eta}$ is of order $\frac{1}{\eta}$, we apply Lemma \ref{lem:densities} to $\eta^2 n^{1,\eta}$ to obtain (\ref{eq:n1}). This concludes the proof of the result.
\end{proof}


\section{Proof of Theorem \ref{th:main}}
\label{sec:proof}

We are now in a position to conclude the proof of our main result. We first obtain an estimate for the remainder term $r^{\eps,\eta}$ in \eqref{eq:expeps}.
\begin{proposition} The different terms involved in the estimate of the remainder term $r^{\eps,\eta}$ are bounded as follows
\begin{equation}\label{eq:estf01}
||f^{0,\eta}||_{H^2_xL^2_vL^\infty_y}\leq C\end{equation}
\begin{equation}\label{eq:estfk}
||f^{k,\eta}||_{H^{4-k}_xL^2_vL^\infty_y}\leq\frac{C}{\eta^k}\quad\mbox{and}\quad ||\overline{v\cdot \nabla_x f^{k,\eta}\psi^{*}}||_{H^{3-k}_xL^2_vL^\infty_y}\leq \frac C{\eta^{k-1}}\, \quad 1\leq k\leq 3. \end{equation}
These result yield that 
$$
||r^{\eps,\eta}||_{L^2_{x,v}}=o(\frac{\eps}{\eta}).
$$
\end{proposition}

\begin{proof}
Let us start with $f^{0,\eta}$. Recall that $f^{0,\eta}(x,y,v)=n^{0,\eta}(x)\psi^\eta(y,v)$
and that
$$
||\psi^\eta||_{L^2_vL^\infty_y}\leq C\quad \mbox{and}\quad ||n^{0,\eta}||_{H^2_x}\leq C.
$$
This gives (\ref{eq:estf01}).
Concerning $f^{1,\eta}=-\chi^\eta(y,v)\cdot\nabla_xn^{0,\eta}(x)+n^{1,\eta}(x)\psi^\eta(y,v)$, we get
$$
||f^{1,\eta}||_{H^3_xL^2_vL^\infty_y}\leq C[ ||n^{0,\eta}||_{H^4_x}||\chi^\eta||_{L^2_vL^\infty_y}+||n^{1,\eta}||_{H^3_x}||\psi^\eta||_{L^2_vL^\infty_y}]
$$
and
$$\begin{array}{rcl}
\dps ||\overline{v\cdot \nabla_x f^{1,\eta}\psi^{*}}||_{H^2_xL^2_vL^\infty_y}&\leq&\dps ||n^{0,\eta}||_{H^4_x}| ||\overline{v\chi^\eta\psi^{*}}||_{H^2_xL^2_vL^\infty_y}+||n^{1,\eta}||_{H^3_x}||\overline{v\psi^\eta\psi^{*}}||_{H^2_xL^2_vL^\infty_y}\leq C.
\end{array}$$
This yields (\ref{eq:estfk}) for $k=1$.
Let us now consider $f^{2,\eta}$. Since $\int_V\int_Y f^{2,\eta}\psi^*(v)d\nu(v)dy=0$, $||T^\eta(f^{2,\eta})||_{H^2_xL^2_vL^\infty_y}\leq \frac{C}{\eta}$  and $|| \int_V T^\eta(f^{2,\eta})\psi^*(v)d\nu(v)dv||_{H^2_xL^\infty_y} \leq C$, we get the first point of (\ref{eq:estfk}) for $k=2$. Indeed, we decompose
$$
f^{2,\eta}=T^{\eta -1} [\int_V T^\eta(f^{2,\eta})\psi^*(v)d\nu(v)dv\psi(y,v)]+T^{\eta -1}[T^\eta(f^{2,\eta})-\int_V T^\eta(f^{2,\eta})\psi^*(v)d\nu(v)dv\psi(y,v)]
$$
and use (\ref{est-t}) and (\ref{est-t-2}), respectively, obtained in the appendix below.


Concerning the second point, we write after defining $\chi^*=Q^{*-1}(v\psi^*)$,
\begin{eqnarray*}
\overline{v\cdot \nabla_x f^{2,\eta}\psi^{\eta*}}&=&\int_V v\cdot\nabla_x f^{2,\eta}\psi^{*}d\nu(v)\,=\, \int_V \nabla_x f^{2,\eta}\cdot Q^*(\chi^*(y,v))d\nu(v)\\
&=&\int_V \nabla_x Q(f^{2,\eta})\cdot\chi^*(y,v)d\nu(v)\\
&=&\int_V\nabla_x (-\eta v\cdot\nabla_y f^{2,\eta}-v\cdot\nabla_x f^{1,\eta}(x,y,v)-f^{0,\eta}+S)\cdot\chi^*(y,v)d\nu(v),
\end{eqnarray*}
so that 
\begin{eqnarray*}
||\overline{v\cdot \nabla_x f^{2,\eta}\psi^{\eta*}}||_{H^1_xL^2_vL^\infty_y}&\leq& C[ \eta||f^{2,\eta}||_{H^{2}_xL^2_vL^\infty_y}+||f^{1,\eta}||_{H^{3}_xL^2_vL^\infty_y}+||f^{0,\eta}||_{H^{2}_xL^2_vL^\infty_y}+||S||_{H^{2}_xL^2_vL^\infty_y}]
\end{eqnarray*}
is bounded by $C\eta^{-1}$.
In a same way, since $\int_V\int_Y f^{3,\eta}\psi^*(v)d\nu(v)dy=0$ and $||T^\eta(f^{3,\eta})||_{H^1_xL^2_vL^\infty_y}\leq \frac{C}{\eta}$and $||\int_V T^\eta(f^{3,\eta})\psi^*(v)d\nu(v)dv||_{H^1_xL^\infty_y} \leq \frac{C}{\eta}$, we get the first point of (\ref{eq:estfk}) for $k=3$.

Finally, we write
\begin{eqnarray*}
\overline{v\cdot \nabla_x f^{3,\eta}\psi^{\eta*}}&=&\int_V v\cdot\nabla_x f^{3,\eta}\psi^{*}d\nu(v)\,=\,\int_V \nabla_x f^{3,\eta}\cdot Q^*(\chi^*(y,v))d\nu(v)\\
&=&\int_V \nabla_x Q(f^{3,\eta})\cdot\chi^*(y,v)d\nu(v)\\
&=&\int_V \nabla_x (-\eta v\cdot\nabla_y f^{3,\eta}-v\cdot\nabla_x f^{2,\eta}(x,y,v)-f^{1,\eta})\cdot\chi^*(y,v)d\nu(v),
\end{eqnarray*}
so that 
\begin{eqnarray*}
||\overline{v\cdot \nabla_x f^{3,\eta}\psi^{\eta*}}||_{L^2_xL^2_vL^\infty_y}&\leq& C[ \eta||f^{3,\eta}||_{H^1_xL^2_vL^\infty_y}+||f^{2,\eta}||_{H^2_xL^2_vL^\infty_y}+||f^{1,\eta}||_{H^1_xL^2_vL^\infty_y}]\\
&\leq &\frac{C}{\eta^2}
\end{eqnarray*}
which is bounded by $C\eta^{-2}$. This concludes the proof of \eqref{eq:estfk} for $k=3$.

In view of the estimate \eqref{eq:estr1} and \eqref{eq:estfk}, we get
$$
||r^{\eps,\eta}||_{L^2_{x,v}}\leq C(\frac{\eps^2}{\eta^2}+\frac{\eps^3}{\eta^3}).
$$ This concludes the proof of the proposition.
\end{proof}

Now, the rest of the proof of the main theorem simply consists of collecting the expression for $f^{\eps,\eta}$ given in \eqref{eq:expeps}, the above estimate for $r^{\eps,\eta}$, the expansions for the densities in section \ref{sec:density} and the expansions for the auxiliary functions $\psi^\eta$ and $\chi^\eta=T^{\eta-1}(v\psi^\eta)$ given in section \ref{sec:aux}. This yields \eqref{eq:main} and concludes the proof of Theorem \ref{th:main}.

\section{Postscript} Guillaume Bal and Marjolaine Puel concluded this work after Naoufel Ben Abdallah passed away the Fifth of July, 2010. We keep the memories of beautiful discussions we started in Montreal and continued in Toulouse.

\appendix
\section{Previous results}
\label{annex:previous}
In this appendix, we improve useful propositions proved in \cite{NBAPuVo}. Their proofs in the case where $\psi^*$ depends on $v$ is a straightforward adaptation of the proofs obtained in \cite{NBAPuVo} when $\psi^*=1$, except for the estimates (\ref{est-t}) and (\ref{est-t-2}) below which are new.

 \begin{proposition}\cite{NBAPuVo}
\label{propl}
Let $L$ be the unbounded operator  on $L^2_{per}(\Rm^d_{y}) $ with domain $H^2_{per}(\Rm^d_{y}) $, defined by 
$$\begin{array}{rcl}
 L(\rho)&=&\dps -\int_V \psi^*(v) v\cdot\nabla_y(Q^{-1}(v\cdot\nabla_y(\psi(y,v)\rho(y))))d\nu(v),\\
 &=&\dps-\mbox{div}_{y}(D \nabla_{y}\rho ) +\nabla_y\cdot( U(y)  \rho )
\end{array}$$
and let  $L^*$ be its  adjoint defined by
\begin{equation}
\label{opl*}
L^*(n) = -\mbox{div}_{y}(D^\top(y) \nabla_{y}n ) - U(y)  \cdot \nabla_{y} n  \quad \mbox{with periodic boundary conditions}.
\end{equation}
The matrix-valued function $D$ and the vector field $U$ are given by 
$$
D(y)=  \int_V \psi^* v\otimes Q^{-1}(v\psi)d\nu(v) ,    \quad \quad U(y)=\int_V \psi^* v\otimes Q^{-1}(v\cdot\nabla_y\psi)d\nu(v).
$$  
The following statements hold:\\
(i) $Im(L) = \{u\in L^2_{per}(\Rm^d_{y}) , \quad s.t. \quad \int_Y u\, dy  = 0\}~.$ \\
(ii)There exists a unique, nonnegative function $\rho(y)  \in H^2_{per} (\Rm^d_{y}) $ such that $L(\rho) = 0$ and $\int_Y \rho \, dy  = 1$~.\\
(iii) For any function $v\in Im(L)$, there exists a unique solution  to $L(u) = v$, $\int_Y u\, dy  = 0$ with $u \in H^2_{per}(\Rm^d_{y}) $. 
This solution will be denoted $u = L^{-1} (v)$, and $L^{-1}$ will be referred to as the pseudo inverse of $L$.\\
(iv) $ Im(L^*) = \{u\in L^2_{per}(\Rm^d_{y}) , \quad s.t. \quad \int_Y u(y)  \rho(y) \, dy  = 0\}~.$ \\
(v) The kernel of $L^*$ is the set of constant functions (in the variable $y$).\\
(vi) For any function $v^*\in Im(L^*)$, there exists a unique solution to $L^*(u^*) = v^*$, $\int_Y u^*(y) \, dy  = 0$
with $u^* \in H^2_{per}(\Rm^d_{y}) $. This solution will be denoted $u^* = L^{*-1} (v^*)$, and $L^{*-1}$ will be referred to as the pseudo-inverse of $L^*$.\\
(vii) $D(y)$,  $U(y)$ and $\rho(y)$ have the following symmetry properties:
\begin{itemize}
\item The diffusion matrix $y\mapsto D(y)$ is even. The flux $y\mapsto U(y)$ is odd. 
The equilibrium function $y\mapsto\rho(y)$ is even.
\item The sets of even and odd functions in $y$ are invariant under $L^{-1}$ and $L^{*-1}$.
\end{itemize}
(viii) The diffusion matrix $D(y)$ is in $C^\infty_{per}(\Rm^d_{y})$ and satisfies
\begin{equation}\label{aaa}
 D(y)\xi\cdot\xi=0 \iff \xi=0 \quad \mbox{and furthermore}\quad D(y)\xi\cdot\xi\geq\beta|\xi|^2~,
\end{equation}
for some constant $\beta>0$.\\
(ix) There exists a constant $C$ such that
\begin{equation}
\label{elliptic}
\left\{
\begin{array}{ll}
\|L^{*-1} v^*\|_{H^2_{per}} \leq C \|v^*\|_{L^2_{per}}, &\quad \|L^{-1} v\|_{H^2_{per}} \leq  C \|v\|_{L^2_{per}},\\[3mm]
\|L^{*-1} v^*\|_{L^2_{per}} \leq C \|v^*\|_{H^{-2}_{per}}, &\quad\|L^{-1} v\|_{L^2_{per}} \leq  C \|v \|_{H^{-2}_{per}}~,
\end{array}
\right.
\end{equation}
for all $v\in Im (L)$ and $v^*\in Im(L^*)$.

\end{proposition}

\begin{remark}
We need the fact that $\psi^*$ does not depend on $y$ to identify zero as the first eigenvalue of $L^*$ and then of $L$. When $\psi^*$ is allowed to depend on $y$, more complex global equilibria (or possibly the lack of such a global equilibrium) need to be analyzed on $Y$. 
\end{remark}

  \begin{proposition}\label{borne1} 
The operator $T^\eta = \eta v \cdot \nabla_{y} + Q[y](\cdot)$
 is an unbounded operator on $L^2_{per}(\Rm^d_{y}\times V)$ with domain
$$\mathcal{D} = \{u \in L^2_{per}(\Rm^d_{y}\times V),\quad\mbox{such that } v \cdot \nabla_{y} u  \in L^2_{per}(\Rm^d_{y}\times V) ~\}~.$$
Let $T^{\eta^*} = - \eta v \cdot \nabla_{y} + Q^*[y](\cdot)$ be the adjoint of $T^{\eta}$. Then\\
(i)  The kernel of $T^\eta$ is a one dimensional space spanned by a positive function $\psi^\eta(y,v) $ normalized by
$\int_{V\times Y} \psi^\eta \psi^*(v)d\nu(v) dy  = 1$. The function $\psi^\eta$ satisfies 
 \begin{equation}
\psi^\eta (-y,-v) = \psi^\eta(y,v), \quad a.e. \hbox{ in } y \hbox{ and } v~.
\end{equation}
(ii) The range of $T^\eta$ is the set of functions $g \in L^2_{per}(\Rm^d_{y}\times V)$ such that we have $\int_{V\times Y} g(y,v)\psi^*(v) \, d\nu(v) dy = 0$~.\\ 
(iii) The adjoint $T^{\eta*}$ has the same domain $\mathcal{D}$. Its range is the set of functions $g$ such that
$\int_{V\times Y} \psi^\eta(y,v) \, g(y,v) d\nu(v) dy = 0.$ Its kernel is spanned by $\psi^*$.\\ 
 \item  For $g\in L^2_{per}(\Rm^d_{y}\times V)$ satisfying $\int_{V \times Y} g(y,v) \psi^*(v)d\nu(v) dy =0$, there exists a unique function $R^\eta \in \mathcal{D}$ 
such that
\begin{equation}\label{18}
T^\eta(R^\eta)= g\quad \mbox{and}\quad \int_{V\times Y} R^\eta(y,v)\psi^*(v)d\nu(v) dy =0.
\end{equation}
We denote $R^\eta = \left(T^{\eta}\right)^{-1}(g).$ There exists $\eta_0>0$ such that
\begin{equation}
\label{est-t}
||R^\eta||_{H^k(Y,L^2(V))}\leq {C\over \eta^2}||g||_{H^k(Y,L^2(V))}~,~~~~ 0<\eta<\eta_0.
\end{equation}
If in addition, $\int_V g \psi^*(v)d\nu(v)=0$, then
\begin{equation}\label{est-t-2}
||R^\eta||_{H^k(Y,L^2(V))}\leq \frac{C}{\eta}||g||_{H^k(Y,L^2(V))}.
\end{equation}
Moreover, the following symmetry implications hold true
\begin{equation}
\label{sym-t}
\left\{
\begin{array}{lllllll}
\mbox{If }\,  g(-y,-v) = &g(y,v) &a.e.& \mathrm{then}& R^\eta(-y,-v) = &R^\eta(y,v)&a.e.\\[3mm] 
\mbox{If }\,  g(-y,-v) = &- g(y,v) & a.e.& \mathrm{then}& R^\eta(-y,-v) = &- R^\eta(y,v)& a.e. 
\end{array}
\right.
\end{equation}
(iv) For $g^*\in L^2_{per}(\Rm^d_{y}\times V)$ satisfying $\int_{V \times Y} \psi^\eta(y,v) g^*(y,v) d\nu(v) dy =0$ there exists a unique function 
$R^{\eta*}\in \mathcal{D}$ such that 
\begin{equation}\label{18*}
T^{\eta*}(R^{\eta*})= g^*\quad \mbox{and}\quad \int_{V \times Y} R^{\eta*}(y,v) \psi^{\eta}(y,v)d\nu(v) dy=0~.
\end{equation}
We denote $R^{\eta*} =\left( T^{\eta*}\right)^{-1}( g^*).$ There exists $\eta_0>0$ such that
\begin{equation}
\label{est-t*}
||R^{\eta*}||_{H^k(Y,L^2(V))}\leq {C\over \eta^2}||g^*||_{H^k(Y,L^2(V))}~,~~~~ 0<\eta<\eta_0~.
\end{equation}
If moreover, $\int_V g^* \psi d\nu(v)=0$, then
\begin{equation}\label{est-t*-2}
||R^{\eta^*}||_{H^k(Y,L^2(V))}\leq \frac{C}{\eta}||g^*||_{H^k(Y,L^2(V))}.
\end{equation}
Finally, the following symmetry implications hold true
\begin{equation}
\label{sym-t*}
\left\{
\begin{array}{lllllll}
\mbox{If }\,  g^*(-y,-v) = &g^*(y,v) &a.e.& \mathrm{then}& R^{\eta*}(-y,-v) = &R^{\eta*}(y,v)&a.e.\\[3mm] 
\mbox{If }\,  g^*(-y,-v) = &- g^*(y,v) & a.e.& \mathrm{then}& R^{\eta*}(-y,-v) = &- R^{\eta*}(y,v)& a.e. 
\end{array}
\right.
\end{equation}
\end{proposition}

\begin{proof}
{\bf Step 1:} For the proof of estimate (\ref{est-t-2}), we argue by contradiction. Assume that $||\tilde R^\eta||=1$ and $||\tilde g^\eta||\rightarrow_{\eta\rightarrow0} 0$ with
$$
\eta v\cdot\nabla_y \tilde R^\eta+Q(\tilde R^\eta)=\eta \tilde g^\eta.
$$
Decompose $\tilde R^\eta=\gamma^\eta(y)\psi(y,v)+\delta^\eta(y,v)$, we get by the dissipation property
\begin{eqnarray*}
||\delta^\eta||^2_{L^2(Y\times V)}&\leq& C\int_Y\int_V Q(\tilde R^\eta)\frac{\tilde R^\eta}{\psi}\psi^*d\nu(v)dy\\
\\
&\leq&\int_Y\int_V(\eta\tilde g^\eta-\eta v\cdot\nabla_y\tilde R^\eta)\frac{\tilde R^\eta}{\psi}\psi^*d\nu(v)dy\\
\\
&\leq &C(\eta||\tilde R^\eta||_{L^2(Y\times V)}||\tilde g^\eta||_{L^2(Y\times V)}+\eta ||\tilde R^\eta||_{L^2(Y\times V)}^2)
\end{eqnarray*}
which implies that $||\delta^\eta||_{L^2(Y\times V)}\rightarrow 0$ when $\delta$ goes to zero.

On the other hand
$$
\eta v\cdot\nabla_y\delta^\eta+Q(\delta^\eta)= \eta\tilde g^\eta-\eta v\cdot\nabla_y(\gamma^\eta \psi)
$$
and an integration with respect to $v$ against $\psi^*$ gives
$$
\mbox{div}_y\int_V v\delta^\eta\psi^* d\nu(v)=\int \tilde g^\eta \psi^*d\nu(v)=0.
$$
Recalling that $\chi^*=Q^{*-1}(v\psi^*)$, we write
\begin{eqnarray*}
&&\mbox{div}_y\int_V v\delta^\eta \psi^*d\nu(v)= -\mbox{div}_y(\int_V\chi^*Q(\delta^\eta)d\nu(v))
\\
&=&\eta\mbox{div}_y(\int_V\chi^*v\cdot\nabla_y \delta^\eta d\nu(v))+\eta\mbox{div}_y(\int_V\chi^*v\cdot\nabla_y (\gamma^\eta \psi))
-\eta\mbox{div}_y(\int_V\chi^*\tilde g^\eta d\nu(v))
\end{eqnarray*}
which leads to the following elliptic equation for $\gamma^\eta$
$$
L(\gamma^\eta)=\mbox{div}_y(\int_V\chi^*v\cdot\nabla_y \delta^\eta d\nu(v))-\mbox{div}_y(\int_V\chi^*\tilde g^\eta d\nu(v)).
$$
Since the right-hand side goes to zero in $H^{-2}(Y)$, we obtain that $\gamma^\eta-C^\eta\rho^0(y)\rightarrow 0$.
But
\begin{eqnarray*}
C^\eta-\int_Y\gamma^\eta dy&=&-\int_Y L^{-1}(\mbox{div}_y(\int_V\chi^*v\cdot\nabla_y \delta^\eta d\nu(v))-\mbox{div}_y(\int_V\chi^*\tilde g^\eta d\nu(v))) \rightarrow 0.
\end{eqnarray*}
Moreover
$$\begin{array}{rcl}
\dps\int_Y\gamma^\eta(y)dy&=&\dps\int_{V\times Y} \gamma^\eta(y)\psi(y,v)\psi^*(v)dyd\nu(v)=\int_{V\times Y}(R^\eta(y)-\delta^\eta(y,v))\psi^*(v)dyd\nu(v)\\
\\&=&\dps-\int_{V\times Y}\delta^\eta(y,v)\psi^*(v)dyd\nu(v)\longrightarrow 0 \mbox{ when }\eta\rightarrow 0.
\end{array}$$
Finally, we obtained that $\tilde R^\eta\rightarrow 0$ strongly which leads to a contradiction.

\medskip

{\bf Step 2:} Proof of estimate (\ref{est-t}).  We first recall that $||R^\eta||_{L^2(Y,V)}\leq C$. 
A direct application of the estimate  proved in \cite{NBAPuVo} would make us lose two powers of $\eta$ for each derivative.  The proof of \eqref{est-t} requires additional computations.

We prove the result by induction. Assume that for any multi-index $i$ satisfying $|i|<|k|$, we have  $||\partial_y^i R^\eta||\leq C$. To prove that $||\partial^k_yR^\eta||_{L^2(Y,V)}\leq C$, we argue by contradiction. 

First, we note that for any $f$ and any multi-indice $k=(k_1,k_2,\cdots,k_d)$, we have
$$\partial_y^k(Q^*(f))=\sum_{i_1}^{|k_1|}\cdots\sum_{i_d}^{|k_d|}\Pi_{l=1}^d(C^{i_l}_{k_l})\partial^{k_1-i_1}_{y_1}\cdots\partial^{k_d-i_d}_{y_d}Q(\partial_{y_1}^{i_1}\cdots\partial_{y_d}^{i_d}f).
$$
Thus, we write
$$
\eta v\cdot \nabla_y \partial^k_y\tilde R^\eta+Q( \partial^k_y\tilde R^\eta)=\eta^2\partial^k_y\tilde g^\eta-\sum_{|i|=0}^{|k|-1}{\bf C^{i}_{k}}\partial^{k-i}_{y}Q(\partial_{y}^{i}\tilde R^\eta)
$$
where ${\bf C^{i}_{k}}=\Pi_{l=1}^d(C^{i_l}_{k_l})$ and where the sum runs over $i=(i_1,i_2,\cdots, i_d)$ such that $|i|\leq |k|-1$ and $i_l\leq k_l$ for $1\leq l\leq d$.
%
To obtain a contradiction, we choose a renormalization that implies $||\partial^k_y\tilde R^\eta||=1$ and $||\partial^k_y\tilde g^\eta||\rightarrow 0$ and $||\partial_y^i\tilde R^\eta||\rightarrow 0$ for any $i$ satisfying $|i|<|k|$.

We decompose $\partial^k_y\tilde R^\eta$ into two parts as follows
$$
\partial^k_y\tilde R^\eta=\gamma_k^\eta(y) \psi(y,v)+\delta_k^\eta(y,v).
$$
We obtain from the dissipation property that
\begin{eqnarray*}
||\delta_k^\eta||^2_{L^2(Y,V)}&\leq& -C \int_Y\int_V Q(\partial^k_y\tilde R^\eta)\frac{\partial^k_y\tilde R^\eta}{\psi}\psi^*d\nu(v)dy\\
\\
&=&C \int_Y\int_V (\eta^2\partial^k_y\tilde g^\eta-\sum_{|i|=0}^{|k|-1}{\bf C^{i}_{k}}\partial^{k-i}_{y}Q(\partial_{y}^{i}\tilde R^\eta)-\eta v\cdot\nabla_y\partial^k_y\tilde R^\eta)\frac{\partial^k_y\tilde R^\eta}{\psi}\psi^*d\nu(v)dy
\end{eqnarray*}
which gives
$$
||\delta_k^\eta||_{L^2(Y,V)}\longrightarrow_{\eta\rightarrow 0} 0.
$$
On the other hand, $\gamma_k^\eta$ satisfies the following elliptic equation
$$\begin{array}{rcl}
L(\gamma_k^\eta)&=&\dps\int_V\partial^k_y\tilde g^\eta \psi^*+\mbox{div}_y(\int_V\chi^*v\cdot\nabla_y(\delta_k^\eta))-\eta\mbox{div}_y\int_V \partial_y^k\tilde g^\eta\chi^* 
\\
&&\dps+\frac{1}{\eta}\mbox{div}_y\int_V\sum_{|i|=0}^{|k|-1}{\bf C^{i}_{k}}\partial^{k-i}_{y}Q(\partial_{y}^{i}\tilde R^\eta)\chi^*.
\end{array}$$
We have to control the last term, which is a priori of order $\frac{1}{\eta}$.
For that purpose, we write
$$
\int_V \partial^{k-i}_{y}Q(\partial_{y}^{i}\tilde R^\eta)\chi^*=\int_V\partial_{y}^{i}\tilde R^\eta\partial^{k-i}_{y}Q^*(\chi^*).
$$
Moreover, since $\psi^*$ does not depend on $y$
$$\partial_y^k(Q^*(\chi^*))=\sum_{i_1}^{|k_1|}\cdots\sum_{i_d}^{|k_d|}\Pi_{l=1}^d(C^{i_l}_{k_l})\partial^{k_1-i_1}_{y_1}\cdots\partial^{k_d-i_d}_{y_d}Q(\partial_{y_1}^{i_1}\cdots\partial_{y_d}^{i_d}\chi^*)=0
$$
and then
$$
\int_V\sum_{|i|=0}^{|k|-1}{\bf C^i_k}\partial^{k-i}Q(\partial^i\tilde R^\eta)\chi^*=- \int_V\sum_{i=0}^{|k|-1}{\bf C^i_k}\partial^i\tilde R^\eta\sum_{|j|=1}^{|k-i|}{\bf C^j_{k-i}}\partial_y^{k-i-j}Q^*(\partial_j\chi^*)
$$
which after rearranging the sums gives
\begin{eqnarray*}
\int_V \sum_{|i|=0}^{|k|-1}{\bf C^i_k}\partial^{k-i}Q(\partial^i\tilde R^\eta)\chi^*&=&- \int_V\sum_{|j|=1}^{|k|}\partial_j\chi^*\sum_{|i|=0}^{|k-j|}{\bf C^j_{k-i}}{\bf C^i_k}\partial_y^{k-i-j}Q(\partial^i\tilde R^\eta)\\
\\&=&-\int_V\sum_{|j|=1}^{|k|}{\bf C^j_{k}}\partial_j\chi^*\sum_{|i|=0}^{|k-j|}{\bf C^i_{k-j}}\partial_y^{k-i-j}Q(\partial^i\tilde R^\eta).
\end{eqnarray*}
Finally, this leads to
$$
\int_V \sum_{|i|=0}^{|k|-1}{\bf C^i_k}\partial^{k-i}Q(\partial^i\tilde R^\eta)\chi^*=-\int_V\sum_{|j|=1}^{|k|}{\bf C^j_{k}}\partial_j\chi^*\partial_y^{k-j}(Q(\tilde R^\eta))
$$
and hence
$$
\frac{1}{\eta}\mbox{div}_y\int_V \sum_{|i|=0}^{|k|-1}{\bf C^i_k}\partial^{k-i}Q(\partial^i\tilde R^\eta)\chi^*=\frac{1}{\eta}\mbox{div}_y\int_V\sum_{|j|=1}^{|k|}{\bf C^j_{k}}\partial_j\chi^*(-\eta v\cdot\nabla_y\partial_y^{k-j}\tilde R^\eta+\eta^2\partial_y^{k-j}\tilde g^\eta).
$$
This term converges to zero in $H^{-2}(Y,L^2(V))$ when $\eta\to0$ and  leads to a contradiction.
\end{proof}


\end{document}